\documentclass[reqno,english]{amsart}

\usepackage{amsmath,amsfonts,amssymb,graphicx,amsthm,url}
\usepackage[noadjust]{cite}
\usepackage{stmaryrd}
\usepackage{mathrsfs,booktabs,tabularx}
\usepackage{xifthen,xcolor,tikz,setspace}
\usetikzlibrary{decorations.pathmorphing,patterns,shapes,calc,decorations}
\usetikzlibrary{decorations.pathreplacing}
\usepackage{mathtools,upgreek,paralist}
\usepackage[final]{showkeys} 

\definecolor{refkey}{gray}{.75}
\definecolor{labelkey}{gray}{.5}
\usepackage[algo2e,boxed,vlined,algoruled]{algorithm2e}
\usepackage{comment}
\usepackage[shortlabels]{enumitem}

\usepackage[letterpaper]{geometry}
\usepackage[colorinlistoftodos]{todonotes}
\presetkeys{todonotes}{inline, color=green}{}

\usepackage{accents}
\usepackage[algo2e,boxed,vlined,algoruled]{algorithm2e}

\usepackage[small]{caption}
\usepackage[colorlinks=true]{hyperref}
\colorlet{DarkGreen}{green!50!black}
\colorlet{DarkGray}{gray!60!black}

\numberwithin{equation}{section}

\renewcommand{\epsilon}{\varepsilon}

\usepackage{color}
 \definecolor{refkey}{gray}{.5}
 \definecolor{labelkey}{gray}{.5}
\definecolor{light}{gray}{.9}

\usepackage{soul}

\usepackage[capitalise]{cleveref}

\newtheorem{theorem}{Theorem}[section]
\newtheorem*{theorem*}{Theorem}
\newtheorem{lemma}[theorem]{Lemma}

\crefname{claim}{Claim}{Claims}

\newtheorem{prop}[theorem]{Proposition}

\newtheorem{observation}[theorem]{Observation}

\newtheorem{corollary}[theorem]{Corollary}

\theoremstyle{definition}{

\newtheorem{definition}[theorem]{Definition}

\newtheorem*{definition*}{Definition}

\newtheorem{remark}[theorem]{Remark}
\newtheorem*{remark*}{Remark}

}

\newcommand{\N}{\mathbb N}
\renewcommand{\P}{\mathbb P}

\newcommand{\Z}{\mathbb Z}

\newcommand{\cB}{\ensuremath{\mathcal B}}
\newcommand{\cC}{\ensuremath{\mathcal C}}

\newcommand{\cE}{\ensuremath{\mathcal E}}

\newcommand{\cT}{\ensuremath{\mathcal T}}

\newcommand{\cZ}{\ensuremath{\mathcal Z}}

\newcommand{\Qtwo}{\mathsf{Q}_{\geq 2}}

 \renewcommand{\epsilon}{\varepsilon}

\DeclareMathOperator{\dist}{dist}

\newcommand{\Int}{{\mathsf{Int}}}

\newcommand{\wP}{\ensuremath{\widetilde \P}}

\makeatletter
\crefname{step}{Step}{Steps}
\crefname{case}{Case}{Cases}
\newcommand{\superimpose}[2]{%
  {\ooalign{$#1\@firstoftwo#2$\cr\hfil$#1\@secondoftwo#2$\hfil\cr}}}
  
\newcommand{\sbullet}{%
  \hbox{\fontfamily{lmr}\fontsize{.4\dimexpr(\f@size pt)}{0}\selectfont\textbullet}}

\makeatother

\begin{document}

\title{Critical wetting in the (2+1)D solid-on-solid model}

\author{Joseph Chen}
\address{J.\ Chen\hfill\break
Courant Institute\\ New York University\\
251 Mercer Street\\ New York, NY 10012, USA.}
\email{jlc871@courant.nyu.edu}

\author{Reza Gheissari}
\address{R.\ Gheissari\hfill\break
Department of Mathematics \\ Northwestern University }
\email{gheissari@northwestern.edu}

\author{Eyal Lubetzky}
\address{E.\ Lubetzky\hfill\break
Courant Institute\\ New York University\\
251 Mercer Street\\ New York, NY 10012, USA.}
\email{eyal@courant.nyu.edu}

\begin{abstract}
\vspace{-0.2cm}
In this note, we study the low temperature $(2+1)$D SOS interface above a hard floor with critical pinning potential $\lambda_w= \log (\frac{1}{1-e^{-4\beta}})$. At $\lambda<\lambda_w$ entropic repulsion causes the surface to delocalize and be rigid at height $\frac1{4\beta}\log n+O(1)$; at $\lambda>\lambda_w$ it is localized at some $O(1)$ height. 
We show that at $\lambda=\lambda_w$, there is delocalization, with rigidity now at height $\lfloor \frac1{6\beta}\log n+\frac13\rfloor$, confirming a conjecture of Lacoin.
\end{abstract}

\maketitle
\vspace{-0.9cm}
\section{Introduction}
The Solid-On-Solid (SOS) model above a wall with an attractive pinning force of $\lambda>0$ along the wall, is the following probability distribution over nonnegative height functions $\phi$ over $\Lambda_n =  \{- \lfloor \frac{n}{2}\rfloor ,\ldots,\lfloor \frac{n}{2}\rfloor\}^2$: 
\begin{align}\label{eq:SOS-measure}
    \P(\phi) \propto \exp\Big( - \beta \sum_{x\sim y} |\phi_x - \phi_y|  + \lambda \sum_{x} \mathbf{1}_{\phi_x = 0}\Big)\,, \qquad \phi:\Lambda_n \to \Z_{\ge 0}\,.
\end{align}
A standard setting for the model sets $0$ boundary condition on $\Z^2 \setminus \Lambda_n$, which is incorporated into~\eqref{eq:SOS-measure} by viewing $\phi$ as extended onto $\Lambda_n$ together with its outer vertex boundary $\partial \Lambda_n$, but forced to be $0$ on $\partial \Lambda_n$. 

When $\lambda =0$ and $\beta$ is large (low temperatures), there is a competition between rigidity of the interface, and entropic repulsion from the constraint that $\phi_v \ge 0$ for all $v$. Bricmont, El Mellouki and Fr\"ohlich~\cite{BEF86} showed that the typical height of the interface is of order $\log n$, and the works~\cite{CLMST14,CLMST16} studied the typical height and the shapes of the level curves in detail. In particular, it was shown that the interface rises along the boundary, and is typically rigid about height $\lfloor\frac{1}{4\beta} \log n\rfloor$ (or possibly the preceding integer for certain~$n$).

When $\lambda>0$, the pinning potential competes with the entropic repulsion;
Chalker~\cite{Chalker} showed that for all large $\beta$, there is a critical $\lambda_w(\beta)$ separating a localized regime of $\lambda>\lambda_w$ in which the interface height at the origin is tight: $\phi_o=O_{\textsc p}(1)$ as $n$ grows (and $\phi_x$ is a certain constant $k(\lambda)$ for most $x\in\Lambda_n$), and a delocalized \emph{wetting regime} of $\lambda<\lambda_w$ where $\phi_o\to\infty$ (as does $\phi_x$ for almost all $x\in\Lambda_n$). Since that work, the two regimes have been investigated in more detail, e.g.,~\cite{Miracle-Sole-layering-and-wetting,ADM-layering-SOS}, culminating in works of Lacoin~\cite{Lacoin-SOS1,Lacoin-SOS2}. These last two works identified 
\begin{align}\label{eq:lambda-w}
    \lambda_w = - \log (1-e^{ - 4\beta})\,,
\end{align}
and showed that on the $\lambda>\lambda_w$ side, an infinite sequence of \emph{layering} transitions occur as $\lambda\downarrow \lambda_w$. Namely, there is a sequence of $\lambda_1>\lambda_2>\ldots$ exponentially converging to $\lambda_w$ such that between $\lambda_{i-1}$ and $\lambda_i$, the interface is rigid about height $i$ (independent of $n$), and the $\lambda_i$ mark discontinuous (first-order) jump points for the expected height above the origin, say. For more on this and related phenomena, see the surveys~\cite{Velenik06,IV18}.

Recently, Feldheim and Yang~\cite{FeldheimYang23} showed that with probability $1-o(1)$ (with high probability, or w.h.p.), the typical height remains $\lfloor \frac{1}{4\beta}\log n\rfloor+O(1)$, as is the case for $\lambda=0$, throughout the wetting regime of $\lambda<\lambda_w$. The behavior exactly at the critical point $\lambda=\lambda_w$  remained open. Lacoin conjectured (as stated in~\cite{FeldheimYang23}) that the interface at $\lambda=\lambda_w$ should still be delocalized, but it should be rigid about height $\frac1{6\beta}\log n$ rather than $\frac1{4\beta}\log n$.
Some evidence towards a lower bound was offered in~\cite{FeldheimYang23} (but only if one assumes an a priori bound of $O(n^{4/3})$ on the total number of zeros of the SOS function $\phi$).

In this note, we prove the above conjecture, showing that at the critical wetting point $\lambda_w$, all but an $\epsilon_\beta$ fraction of the sites have height exactly (up to a possible $-1$ for certain values of $n$) 
\begin{align}\label{eq:h-n-*}
    h_n^* = \lfloor \tfrac{1}{6\beta} \log n + \tfrac {1}{3}\rfloor\,.
\end{align}

\begin{theorem}\label{thm:main}
There exists an absolute constant $C$ such that for all $\beta$ large enough and $\lambda = \lambda_w$, w.h.p., 
    \begin{align*}
        |\{x \in \Lambda_n: \phi_x \notin \{h_n^*-1, h_n^*\}\}| \le \frac{C}{\beta} n^2\,.
    \end{align*}
\end{theorem}

\begin{remark}
For ``most'' values of $n$ (e.g., a subset of $\N$ with a natural density of at least $0.6$), we can identify the single height $h\in \{h_n^*-1,h_n^*\}$ such that, w.h.p., all but an $\epsilon_\beta$-fraction of the sites  have height $h$.
For instance, if the fractional part of $\frac1{6\beta}\log n+\frac13$ falls in the interval $[0.34,1)$, then $h=h_n^*$ (see \cref{rem:lower-bound-k=1}), and if it falls in the interval $[0,\frac{\log \beta}{8\beta})$ then $h=h_n^*-1$ (see \cref{rem:upper-bound-k=1}).
\end{remark}

\begin{remark}\label{rem:shape}
   Lemma~\ref{lem:upper-bnd} in fact shows that all connected sets of vertices of height at least $h_n^*+1$ have exponential tails on their sizes, and by a union bound the largest one is $O(\log n)$ with high probability. One would expect, in analogy with the $\lambda=0$ case~\cite{CLMST16}, that there is a unique macroscopic connected component of the $\{x: \phi_x \in \{h_n^*-1,h_n^*\}\}$, occupying all but some $\epsilon_\beta$ fraction of $\Lambda_n$, whose scaling limit is a Wulff shape.
\end{remark}

\begin{figure}
    \centering
    \includegraphics[height=0.25\textwidth]{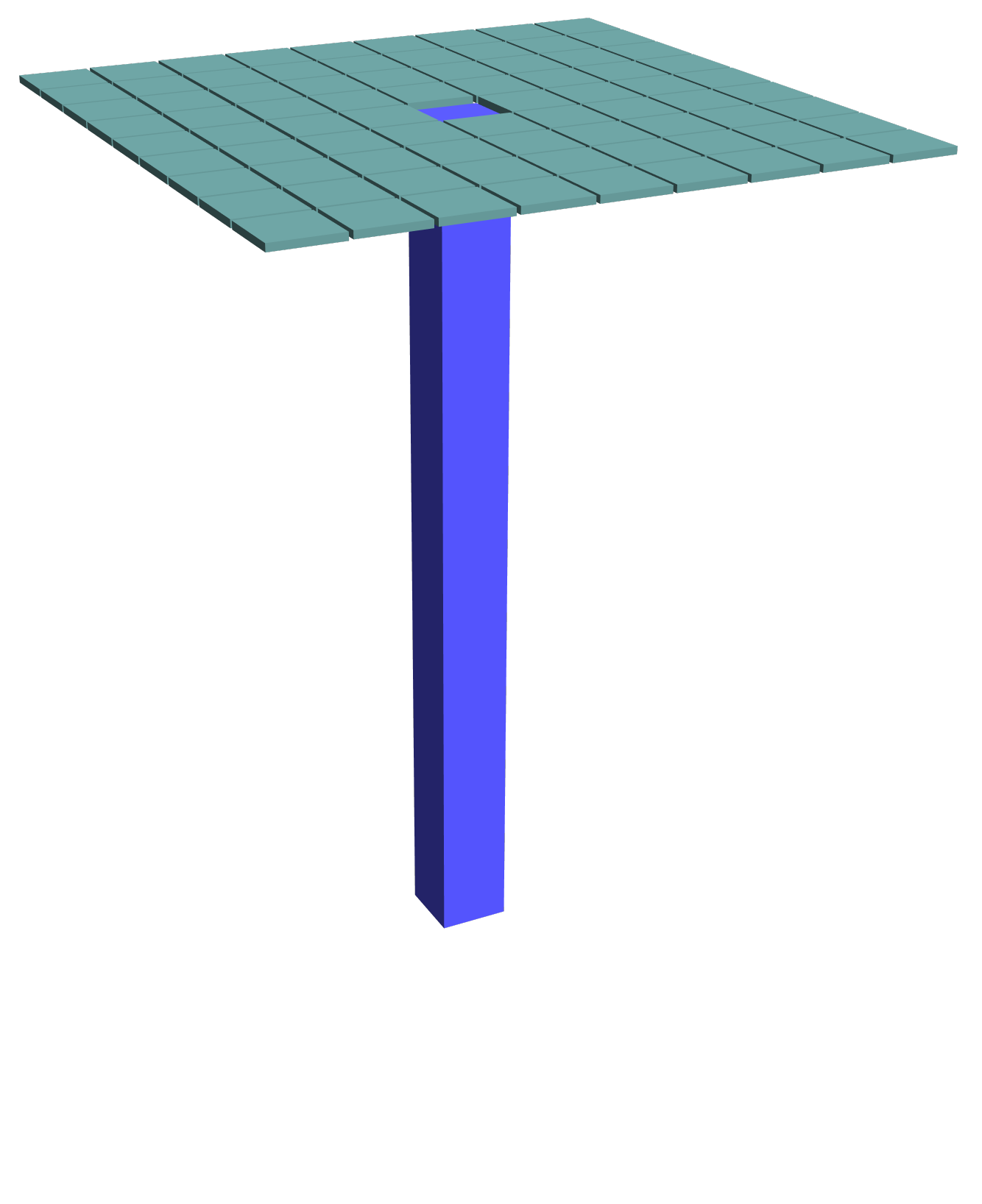}
    \hspace{20pt}
    \includegraphics[height=0.25\textwidth]{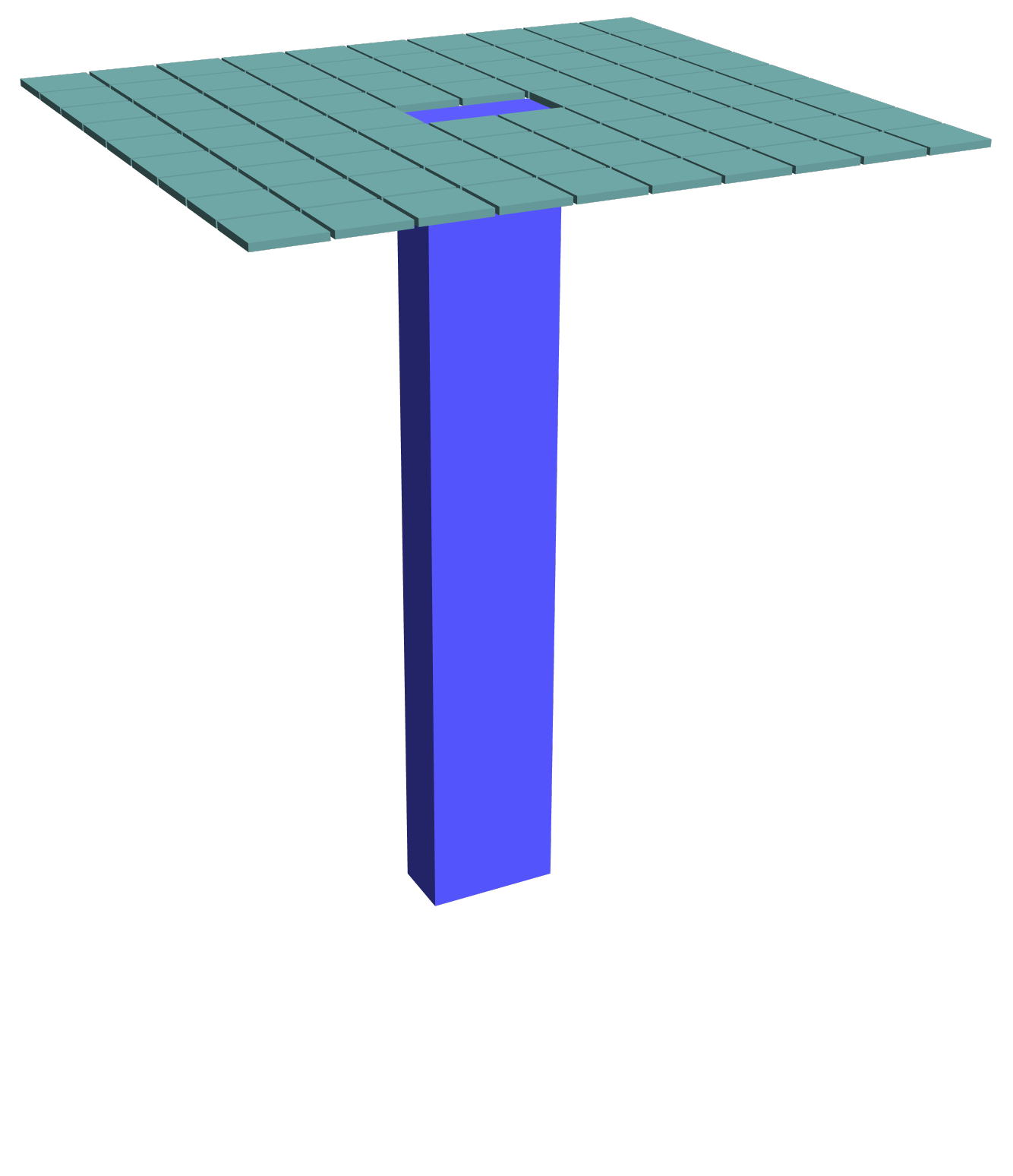}
    \hspace{20pt}
    \includegraphics[height=0.25\textwidth]{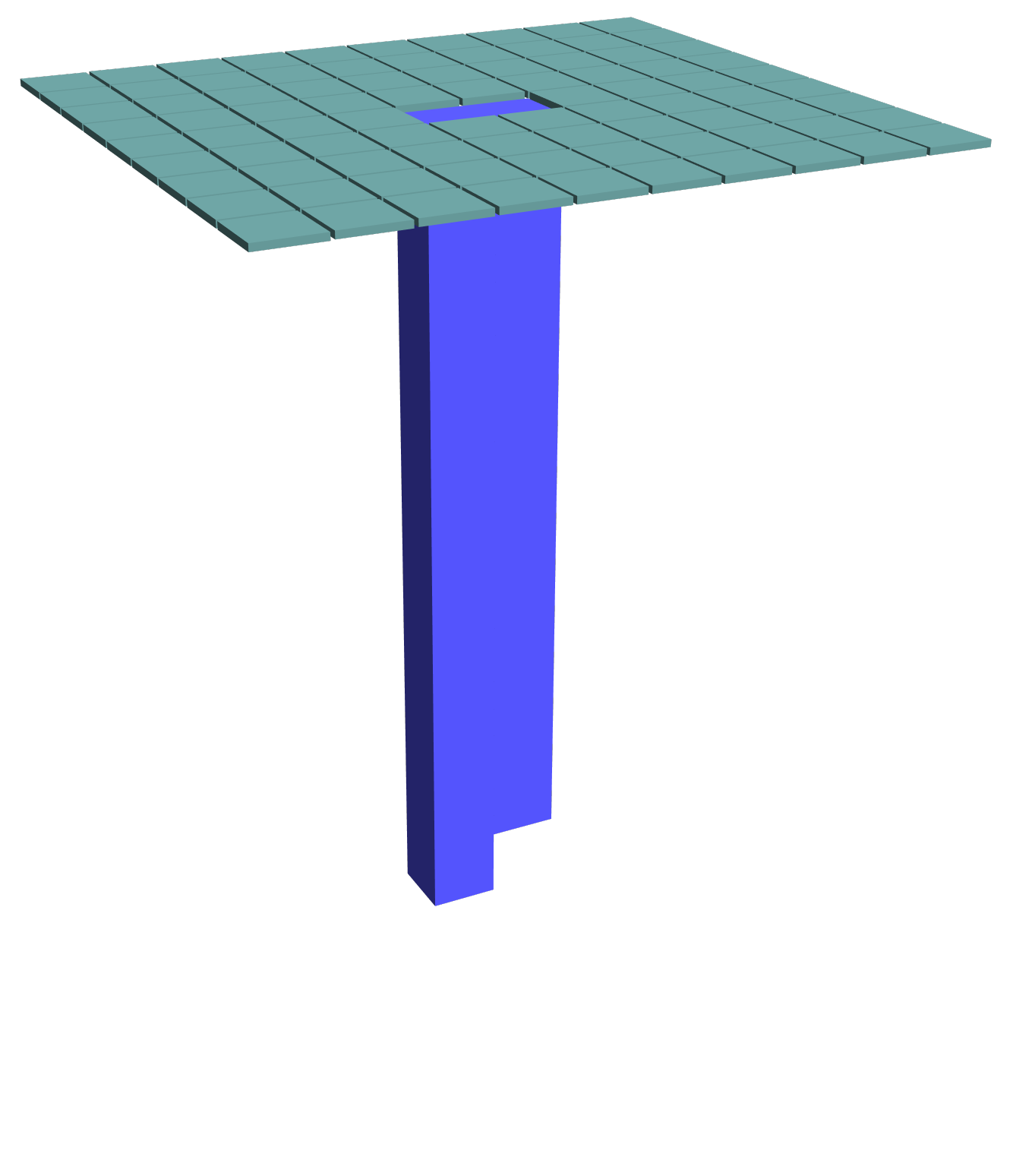}
    \vspace{-0.25in}
    \caption{Downward large deviations of the SOS surface. Left: a $1\times1\times h$ spike of depth $h$ with rate $e^{-4\beta h}$, governing the no-pinning setting ($\lambda=0$). Middle: a $1 \times 2\times h$ spike with rate $e^{-6\beta h}$. Right: a toothlike-spike with rate $e^{-6\beta h+2\beta}$ that governs the critical pinning setting ($\lambda=\lambda_w$).}
    \label{fig:sos-pin-ld}
    \vspace{-0.2in}
\end{figure}
To understand where this height is coming from, let us recap the intuition given in \cite{BEF86} for the case $\lambda=0$. Making the ansatz that at $\beta$ large, the interface is rigid about a certain height $h(n)$, there is a tradeoff between $\exp ( - 4\beta n)$ 
for lifting the bulk of the interface up from height $h-1$ to $h$, and $(1+e^{ - 4\beta h})^{n^2}$ for the entropy gained by the newfound allowance of the interface to have downward oscillations of height $h$ (the easiest way being a straight $1\times 1\times h$ ``spike'') and still remain nonnegative. These two terms are balanced at $\frac{1}{4\beta}  \log n+O(1)$, where rigidity was established in~\cite{CLMST14} (and then refined to $\lfloor \frac1{4\beta}\log n\rfloor-1,\lfloor \frac1{4\beta}\log n\rfloor$ in \cite{CLMST16}).
The work of~\cite{FeldheimYang23} showed that this tradeoff is still governing the behavior as long as $\lambda<\lambda_w$. 

At the critical $\lambda= \lambda_w$, as conveyed in~\cite[Section 1.6]{FeldheimYang23}, one expects there to be a perfect cancellation of entropic repulsion via singleton ($1\times 1 \times h$) downward spikes with the pinning potential, leading the only gain in entropy from lifting the interface to be from the lower order $2\times 1 \times h$ downward spikes. This changes the above competition to be between $\exp(-4\beta n)$ and $(1+e^{ - 6\beta h})^{n^2}$, which is balanced at $\frac{1}{6\beta} \log n+O(1)$. 
The true threshold we identify of~\eqref{eq:h-n-*} is actually governed by a minor refinement of the above expectation; namely, we show that the relevant entropic repulsion effect is by \emph{toothlike-spikes} that consist of a $2\times 1\times (h-1)$ spike with a last block appended below one of its two bottom vertices: see \cref{fig:sos-pin-ld}. This is what causes the $\frac{1}{3}$ correction to $\lfloor\frac{1}{6\beta}\log n\rfloor$ in~\eqref{eq:h-n-*}. For proving the sharp result of Theorem~\ref{thm:main}, it is important to have identified this mechanism in both the lower bound and upper bound on the typical interface height.

\medskip
    We conclude this section with a few necessary preliminaries on the model. 
Define the state space $\widetilde \Omega$ where the floor constraint is relaxed as 
\[\widetilde \Omega = \{\varphi:\Lambda_n \to \Z\,:\, \nexists  x\sim y\, \text{ with } \varphi_x\le -1 \text{ and } \varphi_y\le 0\}\,.\]
Roughly, that is saying that all oscillations where negative heights are attained have singleton intersections with height zero. Notice that $\widetilde \Omega$ is an increasing subset of the set of all $\varphi: \Lambda_n \to \Z$. 

We let $\Qtwo(\varphi)$ be the sites taking height zero that also have a neighbor taking height zero: 
\begin{align}\label{eq:def-q2}
    \Qtwo(\varphi) = \{x: \varphi_x \le 0 \text{ and } \exists y\sim x \text{ with }\varphi_y \le 0\}\,.
\end{align}
Observe that in $\widetilde \Omega$, vertices of $\Qtwo$ cannot take negative heights, and therefore must take height zero. 

Then we can define an SOS model $\wP$ on $\widetilde \Omega$, which when the positive part is taken, gives exactly the SOS model with pinning potential $\lambda$ of~\eqref{eq:SOS-measure} as follows. 
\begin{align}\label{eq:alternate-measure}
    \wP(\varphi) \propto  \exp\Big( - \beta \sum_{x\sim y} |\varphi_x - \varphi_y| + \lambda |\Qtwo(\varphi)|\Big)\,.
\end{align}
The following observation of~\cite{FeldheimYang23} relating $\max\{0,\varphi\}$ for $\varphi\sim \wP$ to $\phi \sim \P$ will be very useful. 
\begin{observation}\label{obs:relaxed-model}
    Fix any $\beta>0$ and let $\lambda = \lambda_w$. Then for every $\phi: \Lambda_n \to\Z_{\ge 0}$, 
    \begin{align*}
        \P(\phi) = \sum_{\varphi: \max\{0,\varphi\} = \phi} \wP(\varphi)\,,
    \end{align*}
    where the maximum is taken pointwise. 
\end{observation}

Lemma 3.1 of ~\cite{Lacoin-SOS1} also used a relaxation of the floor constraint, but that one was more complicated and for general $\lambda$. The case $\lambda = \lambda_w$  with the state space $\widetilde \Omega$ as suggested in~\cite[Section~4]{FeldheimYang23} is especially clean due to the equivalence between the ability in $\varphi$ for a singleton at height $\le 0$ to take a geometric random variable with mean $e^{ - 4\beta}$ as a negative height, and the token of size $\frac{1}{1-e^{-4\beta}}$ that $\phi$ collects by being at height zero. 

Several of our proofs will go by identifying events for $\varphi$ in terms of events in $\wP$ to which they correspond, and then bounding those events under $\wP$.
 
We conclude below with the definition of up-contours and down-contours, applicable for both $\phi$ and $\varphi$, and refer the reader to \cite[Sec.~3]{CLMST14} for a full description of the contour representation of the SOS model. 

\begin{definition}
   For each height $h$, begin by letting $\Gamma_h(\varphi)$ denote the set of all dual-edges that separate $v\sim w$ for $\varphi_v \ge h$ and $\varphi_w \le h-1$. In order to uniquely decompose this into a collection of loops, it is important to have some canonical splitting rule when four dual-edges are incident a dual-vertex. The standard procedure is to round these corners according to a $\textrm{SE}$-$\textrm{NW}$ rule, whence we are left with a collection of disjoint contours $\gamma_1,\ldots,\gamma_m$ with interiors $\Int(\gamma_1),\ldots,\Int(\gamma_m)$ such that $\varphi_v\ge h$ on the interior vertex boundary of $\Int(\gamma_i)$ and $\varphi_v \le h-1$ on the exterior vertex boundary of $S_i$ for all $i$. Then $\gamma_1,\ldots,\gamma_m$ are the up $h$-contours of $\varphi$. 

   Analogously, one uniquely defines the down $h$-contours of $\varphi$ with each one having $\varphi_v \le h$ on its internal vertex boundary, and $\varphi \ge h+1$ on its external vertex boundary.   
\end{definition} 

In what follows, we call a set $S\subset \Lambda_n$ simply-connected if it can be the interior of a contour with the above $\textrm{SE}$-$\textrm{NW}$ splitting rule. For such a set, we use $\partial S$ for its bounding contour. We let $\P_S^h$ (respectively, $\wP^h_S$) denote the SOS measure~\eqref{eq:SOS-measure} (resp.~\eqref{eq:alternate-measure}) on the region $S$ with boundary conditions $h$ on $S^c$.

\subsection*{Acknowledgements}
The authors thank the anonymous referee for helpful comments. 
The research of R.G. is supported in part by NSF DMS-2246780. 
The research of E.L.\ was supported by the NSF grant DMS-2054833.

\section{Upper bound}\label{sec:UB}

In this section we show the following upper bound on the typical height of $\varphi$. 

\begin{prop}
    \label{prop:main-upper-bound}
  There exists $C>0$ such that for all $\beta$ large, we have 
\[ \wP\left( \#\{x: \varphi_x \geq h_n^* + 1\} \geq C e^{- \beta} n^2\right) \leq e^{ - n^{3/4}}\,. 
\]
\end{prop}

The argument goes by showing that up-contours $\gamma$ at height $h_n^* +1$ have exponential tails. This requires us to be able to shift down the interior of such a contour to gain a factor of $e^{ \beta |\gamma|}$ while remaining in the admissible space $\widetilde \Omega$. In turn, in order to argue that the gain from the deletion of the up-contour dominates the entropic repulsion, our goal is to establish that the probability of the interior of $\gamma$ being such that its shift down by $1$ is admissible has probability given by $(1-e^{-6\beta h + 2\beta})^{|\Int(\gamma)|}$.

\subsection{A lifting map to gain entropy}

When trying to identify the rate for downward oscillations with $e^{-6\beta h +2\beta}$, the possible pinning effect of the floor makes Peierls maps that delete down-contours more delicate than they would be when $\lambda=0$---although having a large down-contour has a cost proportionate to the length of the contour, it allows for the possibility of collecting more ``tokens" of $\lambda$ through the $\Qtwo$ term in~\eqref{eq:alternate-measure}. As the number of tokens gained can be proportionate to the area inside the contour, the benefit of the tokens may outweigh the cost of the downward contour for large contours. Hence, we need an a priori upper bound on the size of $\Qtwo$ interior to a contour. 

In what follows, let $S$ be a simply connected set, and let $S' \subseteq S$ be a subset which should be thought of as the interior of a down-contour. Let $W \subseteq S'$ be a further subset representing a set of points on which we want to keep $\varphi$ fixed (ultimately, this will be a pair of adjacent vertices or the empty set).

For any $\varphi$, let $A$ be an arbitrary subset of $\Qtwo(\varphi) \cap S' \setminus W$. For each such $A$, we want to consider the map which raises the height of $\varphi$ inside $S' \setminus (W\cup A)$ by $1$:
$$U_A\varphi = \begin{cases} (U_A\varphi)_z = \varphi_z & z \in (S')^c \cup W \cup A \\ (U_A\varphi)_z = \varphi_z+1 & \text{else}\end{cases}\,.$$
For any vertex set $V$, we can define the edge boundary $\partial_e V$ as the set of edges
\begin{align*}
    \partial_e V := \{\{u,v\}: u\sim v, u \notin V, v \in V\}\,.
\end{align*}
Moreover, for a $\varphi$, define $\Delta S'$ as the change in disagreements attributed to the edge boundary of $S'$:
\begin{equation}
    \Delta S' = \Delta S'(\varphi) := |\{u\sim v: u \in S', v \notin S', \varphi_u \geq \varphi_v\}| - |\{u\sim v : u \in S', v \notin S', \varphi_u < \varphi_v\}|\,.
\end{equation}

\begin{lemma}\label{lem:lifting-map}
    Let $S$ be a simply connected set, $h\ge 0$, $S' \subseteq S$, and $W \subseteq S' \subseteq S$. Let $\varphi$ be such that for all $v \in S$ which are also in the exterior boundary of $S'$ (with respect to $S$), we have $\varphi_v \geq 1$. Then,
    \begin{align*}
        \sum_{A\subseteq \Qtwo(\varphi) \cap S' \setminus W} \wP_S^h(U_A \varphi) \geq \wP_S^h(\varphi)e^{-\beta\Delta S' - \beta|\partial_e W| - \lambda|W|} (1+\frac{1}{2} e^{ - 6\beta})^{|\Qtwo(\varphi) \cap S'\setminus W|/5}\,.
    \end{align*}
\end{lemma}
\begin{proof}
    This proof follows the idea of the proof of \cite[Thm~1.2]{FeldheimYang23}, with some necessary modifications in order to generalize to the above setting. First of all, we have that 
\begin{equation}\label{eq:U_A-weight-change}
    \wP_S^h(U_A\varphi) \geq \wP_S^h(\varphi)\exp{(-\beta\Delta S' - \beta|\partial_e A| - \beta|\partial_e W| - \lambda(|\Qtwo(\varphi)| - |\Qtwo(U_A\varphi)|))}\,.
\end{equation}
Because of the assumption that $\varphi_v \geq 1$ for all $v \in S$ which are in the exterior boundary of $S'$, the isolated zeroes of $\varphi$ in $S \setminus S'$ remain the same no matter how we change $\varphi$ inside of $S'$. Hence, we have that 
\begin{equation}
    |\Qtwo(\varphi)| - |\Qtwo(U_A\varphi)| = |\Qtwo(\varphi)\cap S'| - |\Qtwo(U_A\varphi) \cap S'|\,.
\end{equation}
Now consider a tiling $\cT = \{T_i\}$ of the region $\Qtwo(\varphi) \cap S' \setminus W$ by the tiles in \cref{fig:tiles} (and their rotations). (It is a simple geometric fact that any finite subset of $\Z^2$ can be tiled in this way as long all its connected components have size at least~$2$.) 
Let $A_i = A \cap T_i$, and let $\Qtwo(A_i)= \Qtwo(U_A\varphi)\cap A_i$. Note that we always have $\Qtwo(U_A\varphi) \subseteq \Qtwo(\varphi)$, so we can write 
\begin{equation*}
    |\Qtwo(U_A\varphi) \cap S'| = \sum_i |\Qtwo(A_i)| + |\Qtwo(U_A\varphi)\cap W| \geq \sum_i |\Qtwo(A_i)|\,,
\end{equation*} 
and
\begin{equation*}
    |\Qtwo(\varphi)\cap S'| = \sum_{i} |T_i| + |\Qtwo(\varphi) \cap W| \leq \sum_{i} |T_i| + |W|\,.
\end{equation*}

\begin{figure}
    \begin{tikzpicture}

    \begin{scope}[scale=0.75]

    \node[circle,scale=0.4,fill=gray] (u1) at (0,0) {};
    \node[circle,scale=0.4,fill=gray] (v1) at (1,0) {};
    \draw [black] (u1)--(v1);

    \node[circle,scale=0.4,fill=gray] (u2) at (2,0) {};
    \node[circle,scale=0.4,fill=gray] (v2) at (3,0) {};
    \node[circle,scale=0.4,fill=gray] (w2) at (4,0) {};
    \draw [black] (u2)--(v2)--(w2);

    \node[circle,scale=0.4,fill=gray] (u3) at (5,0) {};
    \node[circle,scale=0.4,fill=gray] (v3) at (6,0) {};
    \node[circle,scale=0.4,fill=gray] (w3) at (6,-1) {};
    \draw [black] (u3)--(v3)--(w3);
    

    \node[circle,scale=0.4,fill=gray] (u4) at (7,0) {};
    \node[circle,scale=0.4,fill=gray] (v4) at (8,0) {};
    \node[circle,scale=0.4,fill=gray] (w4) at (9,0) {};
    \node[circle,scale=0.4,fill=gray] (x4) at (8,-1) {};
    \draw [black] (u4)--(v4)--(w4);
    \draw [black] (v4)--(x4);
    

    \node[circle,scale=0.4,fill=gray] (u5) at (10,0) {};
    \node[circle,scale=0.4,fill=gray] (v5) at (11,0) {};
    \node[circle,scale=0.4,fill=gray] (w5) at (12,0) {};
    \node[circle,scale=0.4,fill=gray] (x5) at (11,-1) {};
    \node[circle,scale=0.4,fill=gray] (y5) at (11,1) {};
    \draw [black] (u5)--(v5)--(w5);
    \draw [black] (x5)--(v5);
    \draw [black] (v5)--(y5);
    

    \end{scope}
    
    \end{tikzpicture}
    \caption{The five tiles that together can cover any realization of $\Qtwo$. }
    \label{fig:tiles}
\end{figure}
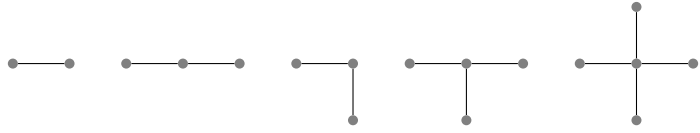

Combining the above, and summing~\eqref{eq:U_A-weight-change} over $A \subseteq \Qtwo(\varphi) \cap S' \setminus W$, we then have
\begin{align}\label{eq:sum-over-Vi}
    \sum_{A \subseteq \Qtwo(\varphi) \cap S' \setminus W} \wP_S^h(U_A\varphi) \geq \wP_S^h(\varphi)e^{-\beta \Delta S' -\beta|\partial_e W|-\lambda|W|}\prod_{i=1}^{|\cT|}e^{-\lambda|T_i|}\sum_{A_i \subseteq T_i}e^{-\beta|\partial_e A_i| + \lambda|\Qtwo(A_i)|}\,,
\end{align}
since summing over $A \subseteq \Qtwo(\varphi) \cap S' \setminus W$ is the same as summing over subsets of each of the covering tiles, and if $A = \bigcup_i A_i$, then $|\partial_e A| \leq \sum_i |\partial_e A_i|$. Furthermore, for each of the five possible choices (up to isometry) of the tile $T$, it was calculated in~\cite[Lemma 4.1]{FeldheimYang23} that 
\begin{equation}
    e^{-\lambda|T|}\sum_{A \subseteq T}e^{-\beta|\partial_e A| + \lambda|\Qtwo(A)|} \geq 1 + \frac{1}{2}e^{-6\beta}\,.
\end{equation}
Plugging this bound into \cref{eq:sum-over-Vi}, and noting that the number of tiles in $\cT$ is at least $|\Qtwo(\varphi) \cap S'\setminus W|/5$ (since the maximum number of vertices in a tile is 5), we have that 
\begin{equation*}
     \sum_{A \subseteq \Qtwo(\varphi) \cap S' \setminus W} \wP_S^h(U_A\varphi) \geq \wP_S^h(\varphi)e^{-\beta  \Delta S' -\beta|\partial_e W|-\lambda|W|}(1+\frac{1}{2}e^{-6\beta})^{|\Qtwo(\varphi) \cap S'\setminus W|/5}\,.\qedhere
\end{equation*}
\end{proof}
\begin{lemma}\label{lem:uniqueness-of-lifting-map}
    For any $\varphi, \varphi'$ such that $\varphi$ and $\varphi'$ are $\geq 1$ on the exterior boundary of $S'$, and any $A \subseteq \Qtwo(\varphi) \cap S' \setminus W$, $A' \subseteq \Qtwo(\varphi') \cap S' \setminus W$ such that either $\varphi \neq \varphi'$ or $A \neq A'$, we have
        \begin{equation*}
            U_A\varphi \neq U_{A'}\varphi'.
        \end{equation*}
\end{lemma}
\begin{proof}
It suffices to show that we can recover $\varphi$ from $U_A\varphi$. We first show how we can recover the set $A$ given $U_A\varphi$. Let $\cZ$ be the set of zeroes of $U_A\varphi$ inside $S'\setminus W$. Observe that $\cZ$ is equal to the disjoint union of $A$ and the set $\{z \in \cZ : \varphi_z = -1\}$. We claim that $A$ is equal to the set
\[\cZ_1 := \{z \in \cZ: \exists v \in S' \setminus W, v \sim z, U_A\varphi_v \in \{0, 1\}\} \cup \{z \in \cZ: \exists v\in W, v \sim z, U_A\varphi_v = 0\}\,.\]

To show $A \subseteq \cZ_1$, note that for all $z \in A$, there must be some $v \sim z$ such that $\varphi_v = 0$ (since $A \subseteq \Qtwo(\varphi)$). As $\varphi$ is $\geq 1$ on the exterior boundary of $S'$, such $v$ must be also be in $S'$. If $v \in W$, then we know $U_A\varphi_v = \varphi_v = 0$. If $v \notin W$, then the fact that $U_A\varphi$ either keeps the value of $\varphi$ or increases it by 1 implies that $U_A\varphi_w \in \{0, 1\}$.

We show that $\cZ_1 \subseteq A$ by showing that $\{z \in \cZ : \varphi_z = -1\} \subseteq \cZ \setminus \cZ_1 =: \cZ_2$. First note that we can write 
\begin{align*}
    \cZ_2 = \{z \in \cZ: \forall v \in S', v \sim z, v \notin W, U_A\varphi_v \geq 2\} \cap \{z \in \cZ: \forall v, v \sim z, v \in W, U_A\varphi_v \neq 0\}\,.
\end{align*}

Observe that for $z$ such that $\varphi_z = -1$, any $v \sim z$ must be such that $\varphi_v \geq 1$ by definition of $\widetilde \Omega$. In particular, we know that $v \notin A$. By definition of $U_A\varphi$, in the case that $v \in S' \cap A^c \cap W^c$, then $U_A\varphi_v = \varphi_v + 1$, and so $U_A\varphi_v \geq 2$. In the case that $v \in W$, then $U_A\varphi_v = \varphi_v \geq 1$, so in particular $U_A\varphi_v \neq 0$. Since $\cZ_1$ only depends on the values of $U_A\varphi$, this shows we can recover the set $A$.

Once we have the set $A$, we can easily recover $\varphi$ from $U_A\varphi$ by taking $\varphi_z = U_A\varphi_z$ if $z \in (S')^c \cup W \cup A$ and taking $\varphi_z = U_A\varphi_z - 1$ otherwise. 
\end{proof}

When $S = S' = \Lambda_n$, $h = 0$, and $W = \emptyset$, Lemma~\ref{lem:lifting-map} corresponds to the result of \cite[Thm~1.2]{FeldheimYang23} that $\Qtwo(\varphi)\leq C_\beta n$ (e.g., this bound holds for the choice $C_\beta = e^{7\beta}$). To make the present paper self-contained ($|\Qtwo(\varphi)| = O(n)$ will be needed in \cref{sec:LB}), we show how to derive said bound as a corollary of the above. 

\begin{corollary}[{\cite[Thm.~1.2]{FeldheimYang23}}]\label{cor:FY-q2+-bound}
Let $\beta$ be large enough, and fix $C_\beta>80 \beta e^{6\beta}$. Then
    \begin{equation*}
        \wP(|\Qtwo(\varphi)| \geq C_\beta n) \leq e^{-n(\frac{C_\beta}{20}e^{-6\beta} - 4\beta)}\,.
    \end{equation*}
\end{corollary}
\begin{proof}
    When $S = S' = \Lambda_n$, $h = 0$, and $W = \emptyset$, we have $\Delta S' = 4n$ and $\wP_{S}^h= \wP_{\Lambda_n}^0 = \wP$. By \cref{lem:lifting-map}, 
    \begin{equation}
        \sum_{A\subseteq \Qtwo(\varphi)} \wP(U_A \varphi) \geq \wP(\varphi)e^{-4\beta n} (1+\frac{1}{2} e^{ - 6\beta})^{|\Qtwo(\varphi)|/5}.
    \end{equation}
    Furthermore, when $S' = S$, the condition that $\varphi$ is $\geq 1$ on the exterior boundary of $S'$ disappears, and we can sum both sides of the above inequality over all $\varphi$ such that $|\Qtwo(\varphi)| \geq C_\beta n$. Then, \cref{lem:uniqueness-of-lifting-map} shows that the left-hand side is upper bounded by 1, and we obtain 
    \begin{equation}
        \wP(|\Qtwo(\varphi)| \geq C_\beta n) \leq e^{4\beta n - \frac{C_\beta}{20}ne^{-6\beta}}\,,
    \end{equation}
    where we use the fact that $1 + x \geq e^{\frac{1}{2}x}$ for $x \in [0, 1]$.
\end{proof}

\subsection{Identifying the rate for downward tooth-like oscillations}
An implication of Lemma~\ref{lem:lifting-map} is that down-contours have exponential tails. Using that, we can show that the rate for having $\varphi_x\le 0,\varphi_y\le 1$ is governed by the minimal weight way to generate this, the tooth-like spike which costs $e^{ - 6\beta h+2\beta}$. 

\begin{lemma}\label{lem:tooth} There exists $\epsilon_\beta>0$ (going to zero as $\beta \to\infty$) such that for every simply connected $S$, every $h\ge 1$, and every pair of adjacent sites $x,y\in S$,
\[ \wP_S^h (\varphi_x \leq 0,\varphi_y \leq 1) \leq (1+\epsilon_\beta)e^{ - 6\beta h + 2\beta}\,.\]
\end{lemma}
\begin{proof}
It suffices to show  
\begin{align}
    \wP_S^h (\varphi_x \leq 1,\varphi_y \leq 1) & \leq (1+\epsilon_\beta)e^{-6\beta (h-1)}\,, \qquad \text{and} \label{eq:pair-reaching-ht-1}\\ 
    \wP_S^h (\varphi_x \leq 0 \mid \varphi_x \leq 1, \varphi_y \leq 1) & \leq (1+\epsilon_\beta)e^{-4\beta}, \label{eq:tooth-given-pair}
\end{align}
To show \cref{eq:pair-reaching-ht-1}, we want to say that with probability $1 - \epsilon_\beta$, conditional on $\varphi_x \leq 1, \varphi_y \leq 1$, the outermost down-contour containing $x$ and $y$ is the loop bounding exactly $\{x,y\}$. To begin, fix a contour $C_{x, y}$ containing $x$ and $y$. Let $\cC_{x, y}(\varphi)$ denote the outermost down-contour containing both $x, y$ in $\varphi$. Let $\cB(C_{x, y})$ be the event that $\varphi_x \leq 1, \varphi_y \leq 1$, and $\cC_{x, y}(\varphi) = C_{x, y}$. Let $S'$ denote the interior of $C_{x, y}$. For every $\psi \in \cB(C_{x, y})$, we know that $\psi$ is $\geq h$ on the exterior boundary of $C_{x, y}$ by definition of $C_{x, y}$ being an outermost down-contour. Hence, we can apply \cref{lem:lifting-map} with $W = \{x, y\}$ and $S' = \Int(C_{x, y})$ to obtain that
\begin{align*}
    \sum_{A \subseteq \Qtwo(\psi)\cap \Int(C_{x, y}) \setminus \{x, y\}} \wP_S^h(U_A\psi) &\geq \wP_S^h(\psi)e^{\beta|C_{x, y}| - 6\beta - 2\lambda}(1+\frac{1}{2}e^{-6\beta})^{|\Qtwo(\psi)\cap \Int(C_{x, y}) \setminus \{x, y\}|/5}\\
    &\geq \wP_S^h(\psi)e^{\beta|C_{x, y}| - 6\beta - 2\lambda}\,.
\end{align*}
Furthermore, note that since $U_A\psi$ did not change the values of $\psi$ on $x, y$, then we still have $U_A\psi_x\leq 1, U_A\psi_y\leq 1$. Hence, summing over $\psi \in \cB(C_{x, y})$ above and applying \cref{lem:uniqueness-of-lifting-map}, we have
\begin{equation}
    \sum_{\psi \in \cB(C_{x, y})}\wP_S^h(\psi)e^{\beta|C_{x,y}|-6\beta - 2\lambda} \leq \sum_{\psi \in \cB(C_{x, y})}\sum_{A\subseteq \Qtwo(\psi)\cap \Int(C_{x, y}) \setminus \{x, y\}}\wP_S^h(U_A\psi) \leq \wP_S^h(\varphi_x \leq 1, \varphi_y\leq 1)\,,
\end{equation}
and in particular
\begin{equation}
     \wP_S^h(\cB(C_{x, y})) \leq e^{-\beta|C_{x,y}| + 6\beta +2\lambda}\wP_S^h(\varphi_x \leq 1, \varphi_y\leq 1)\,.
\end{equation}

Since the number of contours containing $x, y$ with length $l$ is at most $C^l$ for some universal constant $C$, we then have
\begin{align}
    \wP_S^h(\varphi_x \leq 1, \varphi_y \leq 1, |\cC_{x, y}(\varphi)| > 6) &= \sum_{l \geq 8}\sum_{C_{x, y}: |C_{x, y}| = l}\wP_S^h(\cB(C_{x, y})) \\
    &\leq \sum_{l \geq 8}C^le^{-\beta(l-6)+2\lambda}\wP_S^h(\varphi_x \leq 1, \varphi_y\leq 1)\nonumber\\
    &\leq C'e^{-2\beta}\wP_S^h(\varphi_x \leq 1, \varphi_y\leq 1)\,,\nonumber
\end{align}
and in particular that
\begin{equation}\label{eq:dip-is-2-column}
    \wP_S^h(|\cC_{x, y}(\varphi)| = 6 \mid \varphi_x \leq 1, \varphi_y \leq 1) \geq (1-\epsilon_\beta)\,.
\end{equation}
Now, with the above in hand, we claim that
\begin{equation}\label{eq:2-column-dip-probability}
    \wP_S^h(|\cC_{x, y}(\varphi)| = 6, \varphi_x \leq 1, \varphi_y \leq 1) \leq (1+\epsilon_\beta)e^{-6\beta(h-1)}\,.
\end{equation}
First consider the case where $\varphi_x = \varphi_y = 0, |\cC_{x, y}(\varphi)| = 6$. On such configurations, consider the bijective map sending $\varphi$ to $T\varphi$, where  $T\varphi_x = \varphi_x + 6h$, $T\varphi_y = \varphi_y +6h$, and $T\varphi_w = \varphi_w$ for $w \notin \{x, y\}$. The fact that $|\cC_{x, y}(\varphi)| = 6$ implies that for any neighbors $w$ of $\{x, y\}$, we have $\varphi_w \geq h$, so that $\wP_S^h(\varphi) \leq \wP_S^h(T\varphi)e^{-6\beta h}e^{2\lambda}$. Summing over $\varphi$ in this case proves that
\begin{equation*}
    \wP_S^h(|\cC_{x, y}(\varphi)| = 6, \varphi_x =\varphi_y =0) \leq e^{-6\beta h+2\lambda}.
\end{equation*}
Otherwise, if we have that at least one of $\varphi_x$ or $\varphi_y$ is not equal to 0, we consider the map where $T\varphi_x = \varphi_x + 6(h-1), T\varphi_y = \varphi_y + 6(h-1)$, and $T\varphi_w = \varphi_w$ for $w \notin \{x, y\}$. In this case, neither $x$ nor $y$ are in $\Qtwo(\varphi)$, and so we have $\wP_S^h(\varphi) \leq \wP_S^h(T\varphi)e^{-6\beta(h-1)}$. Summing over $\varphi$ in this case proves that
\begin{equation*}
    \wP_S^h(|\cC_{x, y}(\varphi)| = 6, \varphi_x \leq 1, \varphi_y\leq 1, \varphi_x \ne 0 \text{ or } \varphi_y \ne 0) \leq e^{-6\beta(h-1)}.
\end{equation*}
The above two displays prove \cref{eq:2-column-dip-probability}, which combined with \cref{eq:dip-is-2-column} then proves \cref{eq:pair-reaching-ht-1}. 

To prove \cref{eq:tooth-given-pair}, first note that the exact same proof of \cref{eq:dip-is-2-column} yields
\begin{equation}
    \wP_S^h(|\cC_{x, y}(\varphi)| = 6 \mid \varphi_x \leq 0, \varphi_y \leq 1) \geq (1-\epsilon_\beta)
\end{equation}
(simply enforce the starting configurations $\psi$ to satisfy $\psi_x \leq 0, \psi_y \leq 1$, and then apply the same maps $U_A\psi$). Thus, noting that
\begin{align*}
    \wP_S^h(\varphi_x \leq 0 \mid \varphi_x \leq 1, \varphi_y \leq 1) &= \frac{\wP_S^h(|\cC_{x, y}(\varphi)| = 6 \mid \varphi_x \leq 1, \varphi_y \leq 1)}{\wP_S^h(|\cC_{x, y}(\varphi)| = 6 \mid \varphi_x \leq 0, \varphi_y \leq 1)} \wP_S^h(\varphi_x \leq 0 \mid \varphi_x \leq 1, \varphi_y \leq 1, |\cC_{x, y}(\varphi)| = 6)\\
    &\leq (1+\epsilon_\beta)\wP_S^h(\varphi_x \leq 0 \mid \varphi_x \leq 1, \varphi_y \leq 1, |\cC_{x, y}(\varphi)| = 6),
\end{align*}
it suffices to bound the right side above by $(1+\epsilon_\beta)e^{-4\beta}$. Let $\varphi$ be such that $\varphi_x \leq 0, \varphi_y \leq 1, |\cC_{x, y}(\varphi)| = 6$. If $\varphi_x = \varphi_y = 0$, then consider the map which sets $T\varphi_x = T\varphi_y = 1$ and follow the same reasoning as before to obtain 
\begin{equation}\label{eq:zero-zero-case}
    \wP_S^h(\varphi_x = \varphi_y = 0 \mid \varphi_x \leq 1, \varphi_y \leq 1, |\cC_{x, y}(\varphi)| = 6) \leq e^{-6\beta + 2\lambda}.
\end{equation}
If $\varphi_x = k \leq -1$, then $\varphi_y = 1$ by definition of $\widetilde \Omega$. Hence, it remains to consider when $\varphi_x \leq 0, \varphi_y = 1$. In this case, we can consider the map that sets $T\varphi_x = 1$ and obtain that
\begin{equation}
    \wP_S^h(\varphi_x = k \mid \varphi_x \leq 1, \varphi_y\leq 1, |\cC_{x, y}(\varphi)| = 6) \leq e^{-4(1-k)\beta},
\end{equation}
whence summing over $k \leq 0$ and combining with \cref{eq:zero-zero-case} proves that
\begin{equation*}
    \wP_S^h(\varphi_x \leq 0 \mid \varphi_x \leq 1, \varphi_y \leq 1, |\cC_{x, y}(\varphi)| = 6) \leq (1+\epsilon_\beta)e^{-4\beta}.
\end{equation*}
This concludes the proof of \cref{eq:tooth-given-pair} and hence of the lemma.
\end{proof}

\subsection{Tail bounds for up-contours}
Having identified the dominant mechanism for violating the floor constraint of $\wP$ in Lemma~\ref{lem:tooth}, we can show that at heights $h_n^*+1$ and higher, the gain from entropic repulsion is negligible compared to the boundary cost for an up-contour. 
Let $\cC^\uparrow_{S,h}$ denote the event that the boundary of $S$ is a up $h$-contour.
\begin{lemma}\label{lem:upper-bnd}
There exists an absolute constant $c_0>0$ such that, for every simply connected set $S$, 
\[ \wP( \cC_{S,h}^\uparrow) \leq \exp\big( -\beta |\partial S| + c_0 e^{-6\beta h +2\beta} |S|  \big)\,. 
\]
In particular, if $h = h_n^*+1$, then $\wP( \cC_{S,h_n^*+1}^\uparrow) \le \exp(- (\beta - \frac{c_0}{4})|\partial S|)$. 
\end{lemma}

\begin{proof}

For a simply connected domain $S$ (the interior of a contour), let $\cE_S$ denote the event that there does not exist a pair of adjacent sites $x,y\in S$ such that $\varphi_x \le 0$ and $\varphi_y\leq 1$.

\begin{observation}\label{obs:shift-down}
For every simply connected $S$ and integer $h\ge 1$, if $\varphi\in \mathcal C_{S,h}^\uparrow \cap \cE_{S}$, then the configuration $\varphi'$ obtained by shifting $\varphi\mapsto \varphi-1$ in the interior of $S$ keeps the configuration permissible (i.e., $\varphi'\in \widetilde \Omega$). 
\end{observation}

We begin by establishing 
\begin{align}\label{eq:admissibility-given-up-contour}
\wP(\cE_S \mid \cC_{S,h}^\uparrow ) \geq \exp\big(-c_0e^{-6\beta h+2\beta}|S|\big)\,,\end{align}
 as on $\mathcal E_S \cap \cC_{S,h}^\uparrow$ Observation~\ref{obs:shift-down} ensures we can shift down $S$ and gain a weight factor of $e^{ - \beta |\partial S|}$.  
 
Associate to every $\varphi\in\widetilde\Omega$ in the state space of $\wP$, the SOS configuration $\phi = \max\{\varphi, 0\}$ in the state space of $\P$. For every $\varphi\in\cE_S$, we see that $\psi$ satisfies the same property (it does not contain any adjacent pair of sites $x,y$ such that $\phi_x = 0$ and $\phi_y\in\{0,1\}$). In addition, since $h\geq 1$, we have that $\varphi\in\cC_{S,h}^\uparrow$ if and only if $\phi\in \cC_{S,h}^\uparrow$ in its corresponding space.
Thus, it suffices to bound from below
\[ \P\Big(\bigcap_{x\sim y} \big\{\phi_x \leq 0\,, \, \phi_y \leq 1\big\}^c \;\big|\; \cC_{S,h}^\uparrow\Big) \,.\]
By a routine reasoning, when conditioning on $\cC_{S,h}^\uparrow$ we can modify the external boundary of $S$ to $h$ (domain Markov, using that the internal boundary is at least $h$), and then use monotonicity (see, e.g.,~\cite[Sec.~4.1]{Lacoin-SOS1}; it is easy to check that the validity of Holley’s lattice condition is unaffected by the $\lambda$ tokens)
to remove the conditioning that the internal boundary is at least $h$ (as we will look to bound from below an increasing event), so it suffices to bound from below
\[ \P^h_S\Big(\bigcap_{x\sim y} \big\{\phi_x \leq 0\,, \, \phi_y \leq 1\big\}^c \Big) \,.\]
The proof is then concluded from \cref{lem:tooth} by FKG for $\P^h_S$ and the inequality $1 - x \geq \frac12 e^{-x}$ for $x \in [0, 1]$, using the equivalence of $\{\varphi_x\leq 0,\varphi_y\leq 1\}$ for $\P$ and $\wP$ as mentioned above. (The constant $c_0$ comes from enumerating over ordered pairs $(x, y)$ in $S$, and also absorbs the factor of $\frac12$ in the inequality $1 - x \geq \frac{1}{2}e^{-x}$.)

Having shown~\eqref{eq:admissibility-given-up-contour}, let us conclude the proof. We apply the map that shifts $\varphi \mapsto \varphi - 1$ in the interior of $S$ for every $\varphi\in \mathcal C^\uparrow_{S,h} \cap \mathcal E_S$, and obtain that
\begin{equation}\label{eq:admissibility-and-up-contour}
    \wP(\cC_{S,h}^\uparrow \cap \cE_S ) \leq e^{-\beta |\partial S|}\,.
\end{equation}
Dividing this bound by $\wP(\cE_s \mid \cC^\uparrow_{S,h})$ and applying~\eqref{eq:admissibility-given-up-contour}, we conclude the proof. 

When we take $h= h_n^*+1$, the exponential tail follows from $e^{ - 6\beta (h_n^* + 1)+2\beta}\le \frac{1}{n}$ and $|S|\le n|\partial S|/4$. 
\end{proof}

By the same reasoning, we have the following more general version of Lemma~\ref{lem:upper-bnd}.
\begin{corollary}\label{cor:upper-bnd-multiple-loops}
    For any family of disjoint simply connected sets $S_1,\ldots,S_m$ with compatible $\partial S_1,...,\partial S_m$, 
    \begin{align*}
        \wP(\cC_{S_1,h}^\uparrow,\ldots,\cC_{S_m,h}^\uparrow)\le \exp\Big(  - \beta \sum_{i}|\partial S_i| + c_0 e^{ - 6\beta h + 2\beta} \sum_i |S_i|\Big)\,.
    \end{align*}
    In particular, if $h = h_n^*+1$, then $\wP( \bigcap_{i}\cC_{S_i,h_n^*+1}^\uparrow) \le \exp(- (\beta - \frac{c_0}{4})\sum_{i}|\partial S_i|)$. 
\end{corollary}
\begin{proof}
    It is clear by the application of domain Markov that \cref{eq:admissibility-given-up-contour} holds even if we further condition on the values of $\varphi$ outside of $S$. Since shifting $\varphi \mapsto \varphi - 1$ inside of $S$ preserves $\varphi$ outside of $S$, \cref{eq:admissibility-and-up-contour} also holds conditional on $\varphi$ outside of $S$. In particular, \cref{lem:upper-bnd} holds even conditional on having a family of up-contours in the exterior of $S$, whence \cref{cor:upper-bnd-multiple-loops} follows by induction.
\end{proof}

\subsection{Proof of the upper bound}
To conclude Proposition~\ref{prop:main-upper-bound}, by Lemma~\ref{lem:upper-bnd}, we can deduce there are no up $h_n^*+1$ contours of size greater than $\log n$. The following lemma controls the total area confined in smaller contours; this type of estimate is fairly standard once one has the exponential tails of Corollary~\ref{cor:upper-bnd-multiple-loops}.

\begin{lemma}\label{lem:small-area-in-high-contours}
There exists $C, c > 0$ such that for all $\beta$ large, with probability $1-e^{- cn}$, the number of $x\in \Lambda_n$ interior to some up-contour of height $h_n^* +1$ and interior area at most $n^{1.9}$ is $Ce^{-\beta} n^2$. 
\end{lemma}
\begin{proof}
  Let $\Gamma$ denote the set of all outermost $h_n^* +1$ up-contours having interior area at most $n^2/(\log n)^8$. We follow, e.g., the proof of~\cite[Lemma 4.7]{GL-entropic-repulsion} considering the contributions from contours of dyadically increasing sizes. Partition the set of outermost $h_n^* +1$ up-contours into sets $\mathfrak U_1,\mathfrak U_2,\ldots $ given by
\begin{align*}
    \mathfrak U_k = \{\gamma \in \Gamma: 2^{k-1}\le |\Int(\gamma)| \le 2^k\}\,.
\end{align*}
We will show that for a suitable absolute constant $C_0>0$, for each $k=1,\ldots,\lceil \log_2 L_1 \rceil$, 
\begin{align}\label{eq:small-loops-on-scale-k}
    \wP\Big(\sum_{\gamma \in \mathfrak U_k} |\Int(\gamma)|\ge (\epsilon_{\beta,k}n)^2\Big) \le \exp\Big( - (\beta -C) \frac{(\epsilon_{\beta,k}n)^2}{2^{k/2}}\Big)\,, \qquad \text{for}\qquad \epsilon_{\beta,k}:= \frac{C_0}{e^{ \beta/2} k}\,.
\end{align}
A union bound over $k$ implies the claimed bound. 

Fix $k$. To show~\eqref{eq:small-loops-on-scale-k}, suppose $\varphi$ is such that $\sum_{\gamma \in \mathfrak{U}_k} |S_i|\ge (\epsilon_{\beta,k} n)^2$ and suppose the elements of $\mathfrak{U}_k$ are $\gamma_1,\ldots,\gamma_m$ with interiors $S_1,\ldots,S_m$ respectively. By definition of $\mathfrak {U}_k$, it must be the case that 
\begin{align}\label{eq:m-upper-bound}
    m \le \sum_{i=1}^m |S_i| 2^{1-k}\,.
\end{align}
By the isoperimetric inequality in $\mathbb Z^2$, $|\gamma_i|\ge 4\sqrt{|S_i|}$, and the definition of $\mathfrak{U}_k$,  
\begin{align}\label{eq:sum-loop-lengths-lower-bound}
    \sum_{i=1}^{m} |\gamma_i| \ge 4 \sum_{i=1}^{m} |S_i| 2^{-k/2} 
\end{align}

By applying \cref{cor:upper-bnd-multiple-loops} to the $S_i$, we obtain 
\begin{align}\label{eq:multiple-loops-bound}
    \wP\Big( \bigcap_{i} \cC_{S_i,h_n^* +1}^{\uparrow}\Big) 
    \le \exp\Big( - (\beta -c_0/4) \sum_{i}|\gamma_i|\Big)\,.
\end{align}

We union bound over collections $\mathfrak{U}_k$ as follows: letting $\chi$ count $\frac{1}{n^2}\sum_{i=1}^{m}|S_i|$, 
\begin{align*}
    \wP\Big(\sum_{\gamma \in \mathfrak U_k} |\Int(\gamma)| \ge (\epsilon_{\beta,k} n)^2\Big) &  \le \sum_{(\epsilon_{\beta,k}n)^2 \le \chi n^2 \le n^2} \sum_{m\le \chi n^2 2^{1-k}}\binom{n^2}{m} \sum_{L \ge \chi n^2 2^{-\frac{k}{2}+2}} C^L e^{ - (\beta -c_0/4) L} \\ 
    & \le \sum_{\epsilon_{\beta,k}n^2 \le \chi n^2 \le n^2} \sum_{m\le \chi n^2 2^{1-k}}\binom{n^2}{m} e^{ - 4(\beta - C) \chi n^2 2^{-\frac{k}{2}}}\,.
\end{align*}
The first line above used~\eqref{eq:m-upper-bound}--\eqref{eq:sum-loop-lengths-lower-bound} to upper bound $m$ and lower bound $L$, and once the root-points for the $m$ distinct contours have been picked, there are at most $C^L$ ways to generate $m$ associated contours with total length $L$, for some absolute constant $C$. 

Now using the bound $\sum_{j\le \rho N} \binom{N}{j}\le \exp(\mathsf{H}(\rho) N)$ where $\mathsf{H}(\rho)$ is the binary entropy function (this bound uses $\rho\leq1/2$ which holds because otherwise some of the contours could not be distinct), we get 
\begin{align*}
    \wP\Big(\sum_{\gamma \in \mathfrak U_k} |\Int(\gamma)| \ge (\epsilon_{\beta,k} n)^2\Big) \le \sum_{\epsilon_{\beta,k}n^2 \le \chi n^2 \le n^2} \exp\Big(\Big( \mathsf{H}(\chi 2^{1-k}) - (\beta-C) 2^{2-\frac{k}{2}}\chi\Big) n^2 \Big)\,.
\end{align*}
It thus suffices to show that for every $\chi >\epsilon_{\beta,k}^2 = \frac{C_0^2}{e^{\beta} k^2}$, we have 
\begin{align}\label{eq:need-to-show-sfH}
    \mathsf{H}(\chi 2^{1-k})\le 3(\beta -C) \chi 2^{-k/2}\,,
\end{align}
to get~\eqref{eq:small-loops-on-scale-k}, absorbing the pre-factor of $n^2$ from the sum into the $C$ in the exponent. To see this~\eqref{eq:need-to-show-sfH}, using the bound $\mathsf{H}(\rho)\le \rho\log\frac{1}{\rho}+\rho$, and noting that $\chi 2^{1-k} \le (\beta - C) \chi 2^{-k/2}$ for all $k\ge 1$ and all $\beta$ large, we just need to show
\begin{align*}
     2^{1-k} \log \frac{2^{k-1}}{\chi}  \le (\beta - C)2^{1-\frac{k}{2}} \qquad \text{or}\qquad (k-1) (\log 2) + \log \chi^{-1} \le (\beta -C) 2^{k/2}\,.
\end{align*}
By the lower bound on $\chi$, we have $\log \chi^{-1}\le \beta + \log (k^2) - \log C_0^2$. At this point, $C_0$ can be taken large (independent of $\beta$ because the $\beta$ on the left is bounded by $2^{k/2} \beta$ on the right for all $k$) to only consider large values of $k$, and for those it is evident that the right-hand side is larger than the left-hand side. 
\end{proof}

\begin{proof}[\textbf{\emph{Proof of~\cref{prop:main-upper-bound}}}]
 Any vertex $x$ having $\varphi_x \ge h_n^* +1$ must be contained in some up $h_n^* +1$ contour, so it suffices to bound the total area interior to outermost up $h_n^*+1$ contours. By application of Lemma~\ref{lem:small-area-in-high-contours}, it suffices to bound the contribution from up $h_n^*+1$ contours with interior area at least $n^{1.9}$, which necessitates contour length at least $n^{0.95}$. 
 By the fact that there are at most $C^\ell$ many contours of length $\ell$ incident about a vertex $x$ for a universal constant $C>0$, Lemma~\ref{lem:upper-bnd} and a union bound imply that 
 \begin{align*}
     \wP\Big( \bigcup_{\ell \ge n^{0.75}} \bigcup_{\partial S: x\sim \partial S, |\partial S| =\ell} \mathcal C_{S,h_n^*+1}^\uparrow \Big) \le \sum_{\ell \ge n^{0.75}} \sum_{\partial S: x\sim \partial S, |\partial S| =\ell} C^\ell e^{ - (\beta - \frac{c_0}{4})\ell } \le e^{ - (\beta - C') n^{0.75}}\,,
 \end{align*}
which for $\beta$ large, is at most $e^{ - n^{0.75}}$. By a union bound over the $n^2$ possible choices of $x$ in $\Lambda_n$, this rules out the existence of any $h_n^*+1$ up-contours of interior area greater than $n^{3/2}$ concluding the proof. 
\end{proof}

\begin{remark}\label{rem:upper-bound-k=1}
When the fractional part $\xi_n := (\frac{1}{6\beta}\log n + \frac{1}{3}) - \lfloor \frac{1}{6\beta}\log n + \frac{1}{3}\rfloor $ is below a certain threshold (depending on $\beta$), the proof of \cref{prop:main-upper-bound} applies also for $h_n^*$. To see this, notice that in uses of \cref{lem:upper-bnd}, taking $h= h_n^*$, we in fact obtain
\begin{align*}
    \wP\big(\cC_{S,h_n^*}^\uparrow\big) \leq \exp\left(-\beta|\partial S| + c_0 (|\partial S|/4) e^{6\beta\xi_n}\right) \,. 
\end{align*}
Whenever $\xi_n \leq \frac{\log \beta}{8\beta}$, for instance, this is at most $e^{-(\beta/2)|\partial S|}$ and the rest of the steps go through as before.
\end{remark}

\section{Lower bound}\label{sec:LB}
Our goal in this section is to show the following  lower bound on the typical height of $\varphi$. We will then conclude the section by combining it with~\cref{prop:main-upper-bound} and moving back to $\mathbb P$ to deduce Theorem~\ref{thm:main}.  
\begin{prop}
    \label{prop:main-lower-bound}
There exists a constant $C>0$ such that for all $\beta$ large, 
\[ \wP\left( \#\{x: \varphi_x \leq h_n^*-2\} \geq (C/\beta)n^2\right) \leq e^{-\beta n } \,. 
\]
\end{prop}

The proof goes by examining the histogram of $\varphi$ and finding a $k$ such that  the histogram has more faces than it should at height $h_n^* -k$. The map then lifts the interface up by $k$ while injecting entropy through tooth-like spikes of depth $k$ that are now permitted. We work on the event of $|\Qtwo(\varphi)|\le e^{7\beta} n$
which was proved to have high probability by \cite[Thm. 1.2]{FeldheimYang23} (our Corollary~\ref{cor:FY-q2+-bound} for completeness), ensuring that the loss of tokens from lifting $\Qtwo$ up doesn't overwhelm the entropy gained.

\begin{lemma}\label{lem:lower-bound-histogram}
For every $k\ge 2$ define the set
\[ X_k(\varphi) := \{x\sim y:\; \varphi_x = \varphi_y = h_n^*-k\,,\, \mbox{and no $z$ adjacent to $x$ or $y$ has $\varphi_z\leq 0$}\}\,,\]
(where we counted unordered pairs of adjacent sites $x\sim y$). Then for every $k\ge 2$,  
\[ \wP\left( |X_k(\varphi)| \geq  e^{8\beta -5\beta k } n^2 \,,\, |\Qtwo(\varphi)|\leq e^{7\beta}n\right) \leq \exp(-e^{\beta k} n )\,. 
\]    
\end{lemma}
\begin{proof} Fix $k\geq 2$, and let $\mathcal B_k$ be the bad event that $|X_k(\varphi)|\geq e^{8\beta -5\beta k } n^2 $, and $|\Qtwo(\varphi)|\le e^{7\beta} n$. (Note that the choice of $-5\beta k$ is somewhat arbitrary in that $5$ could be any number strictly smaller than $6$, since $e^{-6\beta k} n^2$ is, to first order, the expected number of sites at depth $k$ below $h_n^*$.) 

Let $X'_k(\varphi)$ be an arbitrary subset of disjoint ordered pairs of adjacent sites $(x,y)$ among $X_k(\varphi)$, noting that we may collect at least $|X_k(\varphi)|/7$ such pairs greedily (e.g., by listing the edges $xy$ of $X_k(\varphi)$, ordered $x<y$ lexicographically, proceeding sequentially and collecting each one that is not sharing a vertex with a previously selected edge).

For a prescribed subset of pairs $S\subset X'_k(\varphi)$, $S= \{(x_i,y_i)\}_i$, define the map $\varphi \mapsto T_S \varphi$  via 
\begin{align*}
    T_S \varphi_v = \begin{cases} 0 & \text{$v = x_i$ for some $i$}
    \\ 1 & \text{$v = y_i$ for some $i$} \\ 
    \varphi_v +1 & \text{otherwise}\end{cases}\,,
\end{align*}
i.e., it lifts $\varphi$ by $1$, then adjusts $\varphi_x$ to $0$ and $\varphi_y$ to $1$ for every $(x,y)\in S$ (noting these are all disjoint by construction). 
The fact that $T_S\varphi \in \widetilde \Omega$ follows from the fact that $S$ is not adjacent to any $v$ having $\varphi_v\le 0$, and otherwise the constraint of $\widetilde \Omega$ is increasing, so lifting the rest of the configuration cannot take it outside~$\widetilde \Omega$. 

We can easily compare probabilities 
\begin{align}\label{eq:ratio-of-probabilities}
    \frac{\widetilde{\mathbb P}(T_S \varphi)}{\widetilde{\mathbb P}(\varphi)} \ge \exp\Big( - 4\beta n - (6\beta (h_n^* - k +1) + 2\beta) - \lambda|\Qtwo(\varphi)|\Big)\,.
\end{align}
We now claim that for fixed $k$, the sets of images $T_S\varphi$ across $(\varphi, S)$ are all disjoint, so that we may sum the above expression. 
Note that if $(T_S\varphi_z, T_s \varphi_{z'}) =(0,1)$ for a pair $z\sim z'$, it could not be that $\varphi_z =-1, \varphi_{z'} =0$ because that would violate $\varphi \in \widetilde \Omega$. Therefore, any such pair $(z,z')$ must have both $z,z'$ belonging to edges in $S$. Moreover, all vertices belonging to edges in $S$ get either height $0$ or $1$ and have a neighbor in the other height. In this manner, the set  $S$ can be read off from $T_S\varphi$, and hence the interface $\varphi$ can be read off from $T_S\varphi$ (recover $S$, then set everyone in $S$ back to height $h_n^*-k$ and shift every other site's height down by $1$). Summing over $\varphi\in \mathcal B_k$ and $S\subset X'_k(\varphi)$ and applying~\eqref{eq:ratio-of-probabilities},  
\begin{align*}
1 \ge \sum_{\varphi \in \mathcal B_k} \sum_{S\subset X'_k(\varphi)} e^{ - 6\beta (h_n^* -k+1) + 2\beta} \exp(-4\beta n - \lambda|\Qtwo(\varphi)|) \widetilde {\mathbb P}(\varphi)\,.
\end{align*}
In turn, by binomial theorem, the definition of $h_n^*$~\eqref{eq:h-n-*}, and the  bounds of $|X'_k(\varphi)|\ge \frac{1}{7}|X_k(\varphi)| \ge \frac17 e^{8\beta -5\beta k} n^2 $ and $|\Qtwo(\varphi)| \le e^{7\beta} n$ on $\mathcal B_k$, 
\begin{align*}
1 &\geq \min_{\varphi\in \mathcal B_k} (1+e^{-6\beta (h_n^*-k+1)+2\beta})^{|X'_k(\varphi)|} 
\exp(-4\beta n - \lambda |\Qtwo(\varphi)|)\wP(\mathcal B_k) \\
&\geq  \min_{\varphi \in \mathcal B_k} \exp\Big( \frac1n e^{6\beta (k-1)}|X'_k(\varphi)| 
-4\beta  n - \lambda |\Qtwo(\varphi)|
\Big) \wP(\mathcal B_k) \\ 
&\geq
\exp\Big( \frac{1}{7} e^{\beta( k +2)}n -4\beta  n - \lambda e^{7\beta} n\Big) \wP(\mathcal B_k)\,.
\end{align*}
Thus, using that $\lambda \leq e^{-4\beta}/(1+e^{-4\beta})$, for every $k\geq 2$ we have (for $\beta$ large enough) that
\[ (4\beta +\lambda e^{7\beta}) n   \leq 2e^{3\beta} n \leq (\tfrac17 e^{2\beta} -1) e^{\beta k}n\,,\] and it follows that $\wP(\mathcal B_k) \le \exp( - e^{\beta k} n)$ as required.
\end{proof}

We need to show we aren't losing much of the histogram of $\varphi$ by considering $X_k(\varphi)$ rather than $\varphi^{-1}(h_n^* -k)$. The discrepancy between these two histograms is controlled by the total gradient of a configuration.

\begin{lemma}\label{lem:bound-wall-faces}
  There exists an absolute constant $C>0$ such that, if $\beta$ is large enough then
  \[ \wP\Big(\sum_{x\sim y}|\varphi_x - \varphi_y| \geq (C/\beta) n^2 \Big) \leq \exp(- n^2 )\,.\]
\end{lemma}
\begin{proof}
We can view the surface corresponding to a configuration $\varphi$ as a set of $\sum_{x \sim y}|\varphi_x - \varphi_y|$ vertical faces and $n^2$ horizontal faces in the dual graph to $\Z^3$. Let $C_*>0$ be an absolute constant such that the number of connected components of $k$ faces intersecting a fixed initial point is at most $e^{C_* k}$. Thus, taking this initial point to be somewhere along the boundary of $\Lambda_n$, the number of configurations $\varphi$ such that $\sum_{x \sim y}|\varphi_x - \varphi_y| = k$ is at most $e^{C_*(n^2+k)}$. Considering the trivial map which sends $\varphi \mapsto 0$ everywhere, and summing over $\varphi$ with $k \geq (2C_*/\beta)n^2$, we obtain when $\beta > 3C^*$ that
 \[ \wP(\sum_{x\sim y}|\varphi_x - \varphi_y| \geq (2C_*/\beta) n^2 ) 
 \leq \sum_{k\geq (2C_*/\beta) n^2} e^{C_*(n^2+k)-\beta k} = \frac{e^{C_* n^2- (\beta-C_*)(2C_*/\beta)n^2}}{1-e^{-(\beta-C_*)}} \leq 2\exp(-\tfrac13C_*n^2)\,. \qedhere\]
\end{proof}

\begin{proof}[\textbf{\emph{Proof of~\cref{prop:main-lower-bound}}}]
    We can sum \cref{lem:lower-bound-histogram} for $k \geq 2$ and combine with \cref{cor:FY-q2+-bound} to see that the number of vertices in some pair of $\bigcup_{k\ge 2} X_k(\varphi)$ is at most $e^{-\beta}n^2$ with probability $1-e^{-e^\beta n}$. 
    
    The vertices not in $\bigcup_{k\geq 2} X_k(\varphi)$ are a subset of
    $ \{x : \dist(x,A)\leq 2\}$ where $A$ is the union of $\Qtwo(\varphi)$ along with every $x$ adjacent to a nonzero gradient of $\varphi$. By \cref{cor:FY-q2+-bound} and \cref{lem:bound-wall-faces}, we have $|A|\leq (C/\beta)n^2$ with probability $1-O(\exp(-n^2))$. Noting $|\{x:\dist(x,A)\leq 2\}|\leq 13|A|$ concludes the proof.
\end{proof}

\begin{remark}\label{rem:lower-bound-k=1}
When the fractional part $\xi_n := (\frac{1}{6\beta}\log n + \frac{1}{3}) - \lfloor \frac{1}{6\beta}\log n + \frac{1}{3}\rfloor $ is above a certain (absolute) threshold, the proof of \cref{prop:main-lower-bound} also goes through for $h_n^* -1$. To see this, notice that in the proof of \cref{lem:lower-bound-histogram} we in fact obtained an upper bound on $\wP(\cB_k)$ of the form
\[ \wP(\cB_k) \leq \max_{\varphi\in\cB_k} 
\exp\left(4\beta n + \lambda |\Qtwo(\varphi)| - \frac1n e^{6\beta(k-1)+6\beta \xi_n}|X'_k(\varphi)|\right)\,.
\]
A more careful application of \cref{cor:FY-q2+-bound} (using $C_\beta =90\beta e^{6\beta}$ rather than $e^{7\beta}$ as we used in the proof of \cref{lem:lower-bound-histogram}) allows us to only consider contribution from $\varphi$ with $4\beta n + \lambda |\Qtwo(\varphi)| \leq  100\beta e^{2\beta} n$. Setting $k=1$, whenever we have $\xi_n \geq \frac13+\delta_0$ for some absolute constant $\delta_0>0$, the right hand of our upper bound on $\wP(\cB_1)$  is at most
\[ \max_{\varphi\in\cB_1} \exp\Big(-\Big( e^{(2+6\delta_0)\beta}\frac{|X'_k(\varphi)|}{n^2}-
100\beta e^{2\beta} \Big)n\Big)\,.
\]
Thus, for $\beta$ large enough depending on $\delta_0$, the event that $|X'_k(\varphi)| \geq e^{-\delta_0\beta} n^2$  would have exponentially small probability in $n$ as needed, implying the same up to a factor of $7$ for $|X_k(\varphi)|$.
\end{remark}

\begin{proof}[\textbf{\emph{Proof of Theorem~\ref{thm:main}}}]
    Under the identification $\phi_v = \max\{0,\varphi_v\}$ for all $v$, the sets of sites not at height $\{h_n^* -1,h_n^*\}$ are fixed. Thus, by Observation~\ref{obs:relaxed-model},
    $$\mathbb P(|\{x:\phi_x \notin \{h_n^* -1,h_n^*\}\}| >\tfrac{C}{\beta} n^2) = \wP(|\{x:\varphi_x \notin \{h_n^* -1,h_n^*\}\}| >\tfrac{C}{\beta} n^2)\,,$$ 
    and the result follows by combining~\cref{prop:main-upper-bound} with~\cref{prop:main-lower-bound}. 
\end{proof}

\bibliographystyle{abbrv}
\bibliography{sos_pinning}

\end{document}